\documentclass[draft]{amsart}

\usepackage{amssymb}
\usepackage{url}
\usepackage{amscd}

\theoremstyle{plain}
\newtheorem{theorem}{Theorem}

\newtheorem{proposition}[theorem]{Proposition}

\theoremstyle{remark}
\newtheorem{remark}{Remark}

\def\be#1 {\begin{equation} \label{#1}}
\newcommand{\ee}{\end{equation}}

\def\sqw{\hbox{\rlap{\leavevmode\raise.3ex\hbox{$\sqcap$}}$%
\sqcup$}}
\def\findem{\ifmmode\sqw\else{\ifhmode\unskip\fi\nobreak\hfil
\penalty50\hskip1em\null\nobreak\hfil\sqw
\parfillskip=0pt\finalhyphendemerits=0\endgraf}\fi}

\newcommand{\mb}{\medskip\noindent}
\newcommand{\gb}{\bigskip\noindent}
\newcommand{\R}{\mathbb R}

\newcommand{\s}{\mathcal S}

\newcommand{\M}{\mathcal M}

\begin{document}

\date{November 12, 2010}

\title[Fiber-wise Calder\'on-Zygmund decomposition]{\bf  Fiber-wise Calder\'on-Zygmund decomposition and application 
to a bi-dimensional paraproduct}

\author{Fr\'ed\'eric Bernicot}
\address{Fr\'ed\'eric Bernicot
\\
Laboratoire Paul Painlev\'e \\ CNRS - Universit\'e Lille 1 \\
F-59655 Villeneuve d'Ascq, France}
\email{frederic.bernicot@math.univ-lille1.fr}



\begin{abstract} We are interested in a new kind of bi-dimensional bilinear paraproducts (appearing in \cite{DT}), which do not fit into the setting of bilinear Calder\'on-Zygmund operators. In this paper we propose a fiber-wise Calder\'on-Zygmund 
decomposition, which is  specially adapted to this kind of bi-dimensional paraproduct.
\end{abstract}

\maketitle

\section{Introduction}

Let us first recall what is a standard paraproduct. \\
A bilinear paraproduct $\Pi$ on $\R^d$ is a bilinear operator of the following form~:
$$ \Pi(f,g)(x):= \int_0^\infty \psi_t(f)(x) \phi_t(g)(x) \frac{dt}{t}, $$
where $\psi$ is a smooth function with a spectrum included in a corona around $0$, $\phi$ a smooth function, and for a function $\zeta$, we write the  $L^1$-normalized functions
$$\zeta_{t}(x) := \frac{1}{t^d} \zeta(t^{-1}x).$$
To a function $\zeta_t$, we associate the convolution operator $\zeta_t(f):=\zeta_t \ast f$. So by this way, $\psi_t(f)$ allows us to truncate $f$ in the frequency space around the scale $t^{-1}$. \\
The paraproducts were the first studied singular bilinear operators. Their study began by the works of Bony in 
\cite{bony} and of Coifman and Meyer in \cite{come1,come2,come3}, where in particular first continuities in Lebesgue spaces are shown. Indeed, such paraproducts fit into the setting of bilinear Calder\'on-Zygmund theory (see the work of Grafakos and Torres \cite{GT} and Kenig and Stein \cite{KS}) and so other boundedness can be obtained. Finally, we have boundedness in the full and maximal range of exponents:

\begin{theorem} \label{thm1} For all exponents $1<p,q\leq\infty$ satisfying
$$0< \frac{1}{r}:=\frac{1}{p}+\frac{1}{q}<2,$$
there exists a constant $C$, such that
$$\forall f,g\in\s(\R^d), \qquad \left\| \Pi(f,g) \right\|_{L^r(\R^d)} \leq C \|f\|_{L^{p}(\R^d)}\|g\|_{L^{q}(\R^d)}.$$
\end{theorem}

 When $p,q,r\in(1,\infty)$ the boundedness is an easy consequence of the Littlewood-Paley theory, and if one of the exponents $p,q$ is infinite then the proof of the boundedness requires the notion of Carleson measure. Then, Calder\'on-Zygmund theory allows us to get the boundedness for $r\leq 1$.

\mb  For the last years, Lacey and Thiele have solved the famous Calder\'on's conjecture concerning boundedness of bilinear Hilbert transform \cite{lt1,lt2}. This new kind of bilinear operators (far more singular than the previous one) have then given rise to numerous works. One of the main problems, concerning these operators, which are still open is the generalization of these boundedness for the two-dimensional bilinear Hilbert transforms. \\
Demeter and Thiele have begun to describe a partial result in \cite{DT} but boundedness for the following operator is still open : let $f,g\in{\mathcal S}(\R^2)$
$$ H (f,g)(x,y) := p.v. \int_{\R} f(x-z,y) g(x,y+z) \frac{dz}{z}.$$
There are no results for this model operator. The main difficulty is that the singular operation $p.v. (1/z)$ is simultaneously acting on the two different variables $x$ and $y$. The time-frequency decomposition makes also appear a new phenomenon, mixing the two coordinates. \\
In order to understand first this new problem, we aim to study the following operator~:
$$ T(f,g)(x,y):=\int_0^\infty \psi_{t,x}[f](x,y) \psi_{t,y}[g](x,y) \frac{dt}{t},$$
where $\psi_{t,x}$ and $\phi_{t,y}$ are usual convolutions in the $x$ (or $y$) variable. This operator is to $H$ what the classical paraproduct is to the one dimensional bilinear Hilbert transform. So to study $H$, we have first to understand this new bilinear operator $T$, involving the problem of mixing the two coordinates. This operator is the prototype of the singular operators corresponding to the sixth and remaining open case in \cite{DT}. Note that $T$ can be viewed as a limiting (degenerate) case of some of the other cases (cases 1, 4, 5) in \cite{DT}.

\gb We emphasize that here, we do not prove any boundedness for such operators. It seems that standard arguments (Littlewood-Paley and Calder\'on-Zygmund theories) cannot be employed to prove one boundedness of this new kind of paraproducts and a new approach should be required. However, we propose here an extension result (similar to what is known for classical paraproducts): starting from one supposed boundedness, we then deduce other continuities by making lower the Lebesgue exponents. More precisely, we will prove the following result:

\begin{theorem} \label{thm} Assume that $T$ is bounded from $L^{p_0}(\R^2) \times L^{q_0}(\R^2)$ to $L^{r_0,\infty}(\R^2)$ for some exponents $p_0,q_0\in(1,\infty)$ satisfying
$$\frac{1}{r_0}=\frac{1}{p_0}+\frac{1}{q_0}.$$ 
Then for all exponents $p,q,r$ satisfying $p\in[1,p_0]$, $q\in[1,q_0]$ and 
$$ \frac{1}{r_0}<\frac{1}{r}=\frac{1}{p}+\frac{1}{q}\leq 2,$$
\begin{itemize}
 \item $T$ admits a bounded extension from $L^p(\R^2) \times L^q(\R^2)$ to $L^r(\R^2)$ if $p,q>1$;
 \item $T$ admits a bounded extension from $L^p(\R^2) \times L^q(\R^2)$ to $L^{r,\infty}(\R^2)$ if $p=1$ or/and $q=1$.
\end{itemize}
 \end{theorem}

\mb
We postpone the proof to the next section. Moreover, we will give some comments concerning the dual operators of $T$ too. 

\mb The proof relies on a fiber-wise Calder\'on-Zygmund decomposition. We move the reader to \cite{NOT} for a multiple frequencies Calder\'on-Zgymund decomposition, well-adapted to multi-frequency inequalities. The fiber-wise decomposition, presented here, can be seen as the analogue for the bi-dimensional analysis. Moreover, we point out that Fefferman has already used a more complicated fiber-wise Calder\'on-Zygmund decomposition in \cite{F} to study the double Hilbert transform.

\section{The extension of boundedness result}

The section is devoted to the proof of Theorem \ref{thm}. By bilinear Marcinkiewicz interpolation between weak type inequalities (see \cite{J}), Theorem \ref{thm} is reduced to the following proposition~:

\begin{proposition} \label{prop1} Let assume that $T$ is bounded from $L^{p}(\R^2) \times L^{q}(\R^2)$ to $L^{r,\infty}(\R^2)$ for some exponents $p\geq 1$, $q\geq 1$ and
$$\frac{1}{r}=\frac{1}{p}+\frac{1}{q}.$$ 
Then 
\begin{itemize}
\item if $p>1$, $T$ admits a bounded extension from $L^1(\R^2) \times L^{q}(\R^2)$ to $L^{s,\infty}(\R^2)$ with  
$$ \frac{1}{s}=1+\frac{1}{q};$$
\item if $q>1$, $T$ admits a bounded extension from $L^{p}(\R^2) \times L^{1}(\R^2)$ to $L^{s,\infty}(\R^2)$ with  
$$ \frac{1}{s}=\frac{1}{p}+1.$$
\end{itemize}
 \end{proposition}

\begin{proof} The paraproduct $T$ is almost symmetric in $f$ and $g$. The proof, presented below,  does not distinguish between $\psi$ and $\phi$ and therefore by symmetry, we only prove the first claim for $p>1$. \\
So assume that $p>1$ and let us fix a function $g$, which can be assumed normalized without loss of generality : $\|g\|_{L^{q}(\R^2)}=1$. \\
Let us introduce a well-adapted dense subspace ${\mathcal D}$ of $L^{1}(\R^2)$: ${\mathcal D}$ is the set of functions $f\in L^1(\R^2)$ such that there exist a finite sequence $(f^1_j)_j$ of $L^1(\R)$ and a finite sequence of pairwise disjoint measurable sets $(E_j)_j$ of $\R$ verifying
\be{eq:de} f(x,y) = \sum_j f^1_j(x) {\bf 1}_{E_j}(y).\ee
In fact, ${\mathcal D}$ corresponds to the tensor product $L^1(\R) \otimes L^1(\R)$. \\
We aim to prove that there exists a constant $c$ such that for all $\alpha>0$ and every function $f\in{\mathcal D}$
\be{eq:am}
\left|\left\{ (x,y)\in\R^2,\ |T(f,g)(x,y)|>\alpha\right\}\right| \leq c \alpha^{-s} \|f\|_{L^1(\R^2)}^s. \ee
Using that set ${\mathcal D}$ is dense into $L^1(\R^2)$ and into $L^1(\R^2) \cap L^{p}(\R^2)$, we deduce that $T$ admits an extension which is bounded from $L^1(\R^2) \times L^q(\R^2)$ to $L^{s,\infty}(\R^2)$. \\
Aiming to prove (\ref{eq:am}), we fix $\alpha>0$ and a function $f\in{\mathcal D}$. For every fixed $y\in\R$, by definition of ${\mathcal D}$, $f(\cdot,y)$ is a $L^1(\R)$-function  and so we can choose a Calder\'on-Zygmund decomposition at the scale $\gamma:=\alpha^s \|f\|_{L^1(\R^2)}^{1-s}$ (see \cite{CZ}).
So there exist functions $b_y$ and atoms $(a_{i,y})_i$ supported on cubes $Q_{i,y}$ such that
\begin{itemize}
 \item $f(\cdot,y) = b_y + \sum_{i} a_{i,y};$
 \item for all $y$, we have $\|b_y\|_{L^1(\R)} \leq \|f(\cdot,y)\|_{L^1(\R)}$ and $\|b_y\|_{L^\infty(\R)} \lesssim \gamma$;
 \item for all $i$, the atom $a_{i,y}$ has a vanishing mean value on $Q_{i,y}$ and $\|a_{i,y}\|_{L^1(Q_{i,y})}\lesssim \gamma |Q_{i,y}|$;
 \item for all $y$, $\sum_{i} |Q_{i,y}|\lesssim \gamma^{-1} \|f(\cdot,y)\|_{L^1(\R)}$.
\end{itemize}
 Then, let us define $b(x,y) := b_y(x)$ and $a(x,y):=\sum_{i} a_{i,y}(x)$. It is easy to see that $b$ is in fact measurable in $\R^2$. Indeed using decomposition (\ref{eq:de}) of $f$, we can write
 $$ b(x,y) = \sum_j b^1_j(x) {\bf 1}_{E_j}(y)$$
where $b^1_j$ is the ``good part'' of $f^1_j$ coming from its Calder\'on-Zygmund decomposition. \\
So $b$ is measurable and we have $ \|b\|_{L^\infty(\R^2)} \lesssim \gamma$ and $\|b\|_{L^1(\R^2)}\leq \|f\|_{L^1(\R^2)}$. Thus by interpolation,
 \be{eq:b} \|b\|_{L^p(\R^2)} \lesssim \gamma^{1/p'} \|f\|_{L^1(\R^2)}^{1/p}.\ee
 By splitting, $T(f,g)=T(b,g)+ T(a,g)$, it suffices to study the two terms. \\
 The first one can be bounded using the previous inequality (\ref{eq:b}) and the assumed boundedness of $T$ as follows~:
 \begin{align*}
  \left|\left\{ (x,y)\in\R^2,\ |T(b,g)(x,y)|>\alpha\right\}\right|& \\
  & \hspace{-2cm} \leq \alpha^{-r} \|T(b,g)\|_{L^{r,\infty}(\R^2)}^{r} \lesssim \alpha^{-r} \|b\|_{L^{p}(\R^2)}^r \\
  & \hspace{-2cm} \lesssim \alpha^{-r} \gamma^{r/p'} \|f\|_{L^1(\R^2)}^{r/p} \lesssim \alpha^{-r+sr/p'} \|f\|_{L^1(\R^2)}^{r/p+(1-s)r/p'} \\
  &  \hspace{-2cm} \lesssim \alpha^{-s} \|f\|_{L^1(\R^2)}^s,
 \end{align*}
where we used that
$$ \frac{1}{s}-\frac{1}{r}=\frac{1}{p'}\quad \textrm{and so} \quad \frac{sr}{p'} = r-s.$$
So this term is acceptable according to (\ref{eq:am}). \\
Then, as usual we expect to have pointwise bounds of $T(a_{i,y},g)$  outside the ball $2Q_{i,y}$. So we have first to remove this set. This is possible since
\begin{align*}
\left|\left\{ (x,y),\ x\in \cup_{i} 2Q_{i,y}\right\} \right| & \lesssim \int_\R \left|\cup_i 2Q_{i,y}\right| dy  \lesssim \gamma^{-1}\int_\R \|f(\cdot,y)\|_{L^1(\R)} dy \\
& \lesssim \gamma^{-1}\|f\|_{L^1(\R^2)} \lesssim \alpha^{-s} \|f\|_{L^1(\R^2)}^{s-1+1} \\
& \lesssim \alpha^{-s} \|f\|_{L^1(\R^2)}^s.
\end{align*}
It remains us to study the following term:
$$ I:=\left|\Big\{ (x,y),\ x\in \left(\cup_{i} 2Q_{i,y}\right)^c, \quad |T(a,g)(x,y)| > \alpha  \Big\} \right|.$$
First, we use the specific structure of $T$ to deduce
\begin{align}
 T(a,g)(x,y) & = \int_0^\infty \psi_{t,x}[a](x,y) \phi_{t,y}[g](x,y) \frac{dt}{t} \nonumber \\
 & = \int_0^\infty \sum_{i} \psi_{t}[a_{i,y}](x) \phi_{t,y}[g](x,y) \frac{dt}{t} \label{eq:important}.
 \end{align}
 Then fix $y\in\R$ and take $x\in (\cup_i 2Q_{i,y})^c$, we have to bound the right hand side of (\ref{eq:important}). We denote $r_{Q_{i,y}}$ for the radius of the ball $Q_{i,y}$ and $\M_{y}$ the Hardy-Littlewood maximal function acting on the $y$-variable.  Using the vanishing mean value of $a_{i,y}$ and that $\phi_{t,y}$ is bounded by $\M_y$, it comes with $c_{i,y}$ the center of the ball $Q_{i,y}$
\begin{align*}
 \left|T(a,g)(x,y)\right| & \\
 & \hspace{-1.5cm} \leq \left(\int_0^\infty \sum_i \int_{Q_{i,y}} \left|a_{i,y}(z)\right| \left|\psi_t(x-z)-\psi_{t}(x-c_{i,y})\right| \; dz \frac{dt}{t}\right) \M_{y}g(x,y) \\
 & \hspace{-1.5cm} \leq \left(\int_0^\infty \sum_i \int_{Q_{i,y}} \left|a_{i,y}(z)\right| \frac{r_{Q_{i,y}}}{t^2} \left(1+\frac{|x-c_{i,y}|}{t}\right)^{-M} \; dz \frac{dt}{t}\right) \M_{y}g(x,y) \\
 &  \hspace{-1.5cm} \leq \sum_i \left\|a_{i,y}\right\|_{L^1(Q_{i,y})} \frac{r_{Q_{i,y}}}{|x-c_{i,y}|^2}  \M_{y}g(x,y) \\
 & \hspace{-1.5cm} \leq \gamma \sum_i |Q_{i,y}| \frac{r_{Q_{i,y}}}{|x-c_{i,y}|^2}  \M_{y}g(x,y).
 \end{align*}
Consequently, it comes for $x\in \left(\cup_{i} 2Q_{i,y}\right)^c$
$$  |T(a,g)(x,y)| \lesssim \gamma H(x,y)\M_{y}g(x,y),$$
with 
$$ H(x,y):= \sum_{i}  |Q_{i,y}| \frac{r_{Q_{i,y}}}{|x-c_{i,y}|^2}{\bf 1}_{(2Q_{i,y})^c}(x) .$$ 
Since the function $f\in{\mathcal D}$ can be decomposed with tensor products (see (\ref{eq:de})), the function $H$ can also be written with tensor products which implies that $H$ is measurable.
Since $s^{-1}=1+q^{-1}$, H\"older inequality for weak Lebesgue spaces yields~:
 \begin{align*}
  I & \leq \alpha^{-s} \gamma^s \left\| H\right\|_{L^{1,\infty}(\R^2)}^s  \left\| \M_y[g] \right\|_{L^{q,\infty}(\R^2)}^{s}. 
  \end{align*}
Since $q\geq 1$,  $\M_y$ is of weak type $(q,q)$ and since $g$ is normalized in $L^q(\R^2)$ we have
\begin{align*}
 I & \leq \alpha^{-s} \gamma^s \left\|H\right\|_{L^{1,\infty}(\R^2)}^s \leq \alpha^{-s} \gamma^s \left\|H\right\|_{L^{1}(\R^2)}^s. 
 \end{align*}
Moreover, for all $y\in\R$, we get
\begin{align*}
 \int_\R H(x,y) dx & \leq \sum_{i} |Q_{i,y}| \int_{(2Q_{i,y})^c} \frac{r_{Q_{i,y}}}{|x-c_{i,y}|^2} dx \lesssim \sum_{i} |Q_{i,y}| \\ 
 & \lesssim \gamma^{-1}\|f(\cdot,y)\|_{L^1(\R)},
\end{align*}
which yields $\|H\|_{L^1(\R^2)} \lesssim \gamma^{-1} \|f\|_{L^1(\R^2)}$. Hence, we finally obtain that
\begin{align*}
I & \lesssim \alpha^{-s} \gamma^s (\gamma^{-1}\|f\|_{L^1(\R^2)})^s = \alpha^{-s} \|f\|_{L^1(\R^2)}^s.
   \end{align*}
That exactly corresponds to the desired estimate (\ref{eq:am}) for this term, which ends the proof.
 \end{proof}

\begin{remark} The previous proof is based on a fiber-wise Calder\'on-Zygmund decomposition for $f\in L^1(\R^2)$. We have worked for functions belonging to a dense subspace, in order to get around technical problems of measurability. Indeed for $f\in L^1(\R^2)$, we know that for almost every $y\in\R$, we can consider the function $f(\cdot,y)\in L^1(\R)$. Applying the standard Calder\'on-Zygmund decomposition, we obtain $f(\cdot,y)=b_y+a_y$. Then we define a global ``good part'' of the function $f$, by defining $b(x,y):=b_y(x)$ for almost every $y\in\R$. But to compute estimates on $b$, it is important to first check that $b$ is measurable in $\R^{2}$. In general, this seems to be not obvious, it is related to problems involving measurable selections (in the stopping time argument of the Calder\'on-Zygmund decomposition). \\
That is why, here we have prefered to work with specific functions $f$ which can be split into tensor products and then the measurability is easily obtained.
 \end{remark}

\gb Having concluded the proof of our main result Theorem \ref{thm}, we are now interested in the dual operators of $T$.
Let us compute the two dual operators $T^{*1}$ and $T^{*2}$ (with respect to $f$ or $g$), it comes~:
$$ T^{*1}(h,g)(x,y) = \int_0^\infty \psi_{t,x}[h \phi_{t,y}(g)](x,y) \frac{dt}{t}$$
and 
$$ T^{*2}(f,h)(x,y) = \int_0^\infty \phi_{t,y}[ \psi_{t,x}(f) h ](x,y) \frac{dt}{t}.$$

\gb As we have remarked during the previous proof for (\ref{eq:important}), we have used a specific property of the operator $T$. More precisely, we have employed the fact that to compute $T(f,g)(x,y)$ requires only information on $f(\cdot,y)$ and not on the whole function $f$ (and similarly for $g$), in order that for all $(x,y)\in\R^2$
$$ T(f,g)(x,y) := T(f(\cdot,y),g(x,\cdot))(x,y).$$
This fact is very important in the proof and allowed us to use the standard Calder\'on-Zygmund decomposition. It is interesting to emphasize that this property is not satisfied for the two dual operators $T^{*1}$ and $T^{*2}$. Indeed to compute $T^{*1}(h,g)(x,y)$, we require information on the whole function $g$ (since we have the first operator $\phi_{t,y}$ in the $y$ variable and then the second one $\psi_{t,x}$ in the $x$ variable) and on the function $h$ (we have just an operator $\psi_{t,x}$ on the $x$ variable but a vanishing mean value of $h$ in one of the two variables does not bring an extra decay in order to repeat the previous arguments). So we cannot reproduce the previous reasoning for the dual operators and we do not know if such an extension result is true for them.

\mb The operator $T$ is the prototype operator of those described by the last and sixth case in \cite{DT}. Moreover, we note that the operators appearing in the cases 1,4 and 5 of \cite{DT} do not satisfy this specific structure of $T$ (like the dual operators $T^{*1}$ and $T^{*2}$) too and so cannot be studied by this fiber-wise Calder\'on-Zygmund decomposition.

\end{document}